\documentclass{amsart}
\usepackage{graphicx}
\usepackage{amsmath,bm}
\usepackage{amssymb, comment}
\usepackage[all]{xy}
\usepackage{color}
\usepackage[mathscr]{euscript}

\newcommand{\Ktac}{\operatorname{{\bf K}_{tac}}\nolimits}
\newcommand{\Hom}{\operatorname{Hom}\nolimits}
\renewcommand{\Im}{\operatorname{Im}\nolimits}

\newcommand{\Id}{\operatorname{Id}\nolimits}
\renewcommand{\H}{\operatorname{H}\nolimits}

\newcommand{\rank}{\operatorname{rank}\nolimits}

\newcommand{\n}{\operatorname{\mathfrak{n}}\nolimits}

\newcommand{\comments}[1]{}

\newtheorem{theorem}{Theorem}[section]

\newtheorem{corollary}[theorem]{Corollary}

\newtheorem{lemma}[theorem]{Lemma}
\newtheorem{proposition}[theorem]{Proposition}
\theoremstyle{definition}

\newtheorem{discussion}[theorem]{}
\theoremstyle{definition}

\theoremstyle{definition}

\theoremstyle{definition}
\newtheorem*{example}{Example}
\theoremstyle{definition}

\theoremstyle{definition}

\theoremstyle{definition}

\theoremstyle{definition}

\theoremstyle{definition}
\newtheorem*{remark}{Remark}

\theoremstyle{definition}

\newenvironment{smallarray}[1]
 {\null\,\vcenter\bgroup\scriptsize
  \arraycolsep=.13885em
  \hbox\bgroup$\array{@{}#1@{}}}
 {\endarray$\egroup\egroup\,\null}

\begin{document}

\title{A converse to a construction of Eisenbud-Shamash}

\author{Petter A. Bergh, David A.\ Jorgensen \& W. Frank Moore}

\address{Petter Andreas Bergh \\ Institutt for matematiske fag \\
  NTNU \\ N-7491 Trondheim \\ Norway}
\email{bergh@math.ntnu.no}

\address{David A.\ Jorgensen \\ Department of Mathematics \\ University
of Texas at Arlington \\ Arlington \\ TX 76019 \\ USA}
\email{djorgens@uta.edu}

\address{W. Frank Moore \\ Department of Mathematics and Statistics \\ 
Wake Fores University \\ Winston-Salem, \\ NC 27109  \\ USA}
\email{moorewf@wfu.edu}

\date{\today}

\begin{abstract}
Let $(Q,\n,k)$ be a commutative local Noetherian ring, $f_1,\dots, f_c$ a $Q$-regular sequence in $\n$, and 
$R=Q/(f_1,\dots,f_c)$.  Given a complex of finitely generated free $R$-modules, we give a construction of a complex of finitely generated free $Q$-modules having the same homology.  A key application is when the original complex is an $R$-free resolution of a finitely generated $R$-module. In this case our construction is a sort of converse to a construction of Eisenbud-Shamash yielding a free resolution of an $R$-module $M$ over $R$ given one over $Q$.   
\end{abstract}

\subjclass[2010]{13D02, 13D07}

\keywords{Free resolutions, Eisenbud operators, regular sequences, change of rings}

\thanks{Part of this work was done while we were visiting the Mittag-Leffler Institute in February and March 2015. We would like to thank the organizers of the Representation Theory program.}

\maketitle

\section{Introduction}\label{Sec:Intro}

Let $(Q,\n,k)$ be a commutative local Noetherian ring and $R=Q/I$ for some ideal $I$ of $Q$.  Given a free resolution of an $R$-module $M$, can one then describe a free resolution of $M$ over $Q$?  When $I$ is generated by a $Q$-regular sequence $f_1,\dots, f_c$ contained in $\n$, we give a construction that provides a positive answer to this question.  Our construction is a sort of converse of those of Nagata \cite{Nagata} and Eisenbud-Shamash \cite{Eisenbud} in this same context.  

We have recently learned that similar constructions are given by Eisenbud, Peeva and Schreyer in 
\cite[Section 7]{EisenbudPeevaSchreyer} and Burke in \cite{Burke}

We present the construction in Section 2.  Sections 3-5 consist of some applications.

\section{The Construction}\label{Sec:Construction}

Let $(\overline F,\partial^{\overline F})$ be a complex of finitely generated free $R$-modules.  Denote by $(F,\partial^F)$ a lifting to $Q$ of $(\overline F,\partial^{\overline F})$, that is $F$ is a
graded module consisting of free $Q$-modules, $\partial^F$ is an endomorphism of $F$ of degree $-1$, and 
$F\otimes_Q R=\overline F$ with $\partial^F\otimes R=\partial^{\overline F}$.

\begin{discussion} \label{flat}
We denote by $\Hom_{\text gr}(F,F)=\coprod_{n\le 0}\Hom_{\text gr}(F,F)_n$ the nonpositively graded $Q$-module 
of endomorphisms of $F$, where $\Hom_{\text gr}(F,F)_n=\prod_{i\in\mathbb Z}\Hom_Q(F_i,F_{i+n})$. Since $\Hom_{\text gr}(F,F)_n$ is a direct product of flat modules (and $Q$ is Noetherian), a result of Chase \cite{Chase}
tells us that $\Hom_{\text gr}(F,F)_n$ is a flat $Q$-module.  It follows that $\Hom_{\text gr}(F,F)$ is a 
flat $Q$-module.
\end{discussion}

Let $K$ denote the Koszul complex on $f_1,\dots, f_c$, and $\mathscr{B}$ denote a basis of $K$ together with $0$, that is,
\[
\mathscr{B}=\{e_{i_1}\wedge \cdots \wedge e_{i_j}\mid i_1<\cdots < i_j, 1\le j \le c\}\cup\{0,1\}
\]  
We adhere to the convention that $1\wedge\alpha=\alpha\wedge 1=\alpha$, and assign degrees in the usual way:
$|e_{i_1}\wedge \cdots \wedge e_{i_j}|=j$, $|1|=0$; we set the degree of $0$ to be $-1$.

\begin{lemma}\label{lemma}
There exist (graded) endomorphisms $\{t^\alpha\}_{\alpha\in\mathscr{B}}$ of $F$ of degree $-|\alpha|-1$ such that
for all $0\ne \gamma\in\mathscr{B}$ 
\begin{equation}\label{relation}
\sum_{\alpha\wedge\beta=\pm\gamma}(-1)^{|\beta|+(\alpha\beta)}t^\beta t^\alpha
+\sum_{e_i\nmid\gamma}(-1)^{|\gamma|+(e_i\gamma)}f_it^{[e_i\gamma]}=0
\end{equation}
where $(-1)^{(\alpha\beta)}$ is the appropriate sign such that $(-1)^{(\alpha\beta)}\alpha\wedge\beta=\gamma$, 
and $[e_i\gamma]$ is the element in $\mathscr{B}$ equal to $e_i\wedge\gamma$ up to sign.
\end{lemma}
Note that the first sum in \eqref{relation} involves the terms $t^1t^\gamma$ and $t^\gamma t^1$.

\begin{proof}
We begin by setting $t^0=\Id_F$ and $t^1=\partial^F$.  Then $t^\alpha$ for $|\alpha|\ge 1$ will be defined inductively by degree.

Let $K^\sharp$ denote the augmented Koszul complex $0\to K_c \to\cdots\to K_0 \to R\to 0$, which is an exact complex. Since the $Q$-module $\Hom_{\text gr}(F,F)$ of graded $Q$-linear endomorphisms of $F$ of nonpositive
degree is flat, \ref{flat}, the complex $\Hom_{\text gr}(F,F)\otimes_Q K^\sharp$ is exact.

For degree 1, we have $\partial^{\Hom_{\text gr}(F,F)\otimes K^\sharp}(t^1t^1\otimes 1)=t^1t^1\otimes \overline 1$. Since $t^1=\partial^F$, $t^1\otimes \overline 1=\partial^{\overline F}$, and therefore
$t^1t^1\otimes \overline 1=(\partial^{\overline F})^2=0$. By exactness of $\Hom_{\text gr}(F,F)\otimes K^\sharp$ there exist endomorphisms $t^{e_i}$ of $F$ of degree $-2$ such that 
\[
\partial^{\Hom_{\text gr}(F,F)\otimes K^\sharp}\left(-\sum_{i=1}^c(t^{e_i}\otimes e_i)\right)=t^1t^1\otimes 1
\] 
That is,
$-\sum f_it^{e_i}\otimes 1=t^1t^1\otimes 1$, which gives the desired equation $t^1t^1+\sum f_it^{e_i}=0$.  

Now assume that the $t^\gamma$ have been defined, and satisfy \eqref{relation}, for all $\gamma\in\mathscr B$ with $|\gamma|\le d$. We have
\begin{align*}
\partial^{\Hom_{\text gr}(F,F)\otimes K^\sharp}&\left(\sum_{|\gamma|=d}\left(\sum_{\alpha\wedge\beta=\pm\gamma}(-1)^{|\beta|+(\alpha\beta)}t^\beta t^\alpha\right)\otimes\gamma\right)\\
&=\sum_{|\gamma|=d}\left(\sum_{\alpha\wedge\beta=\pm\gamma}(-1)^{|\beta|+(\alpha\beta)}t^\beta t^\alpha\right)\otimes\partial^K(\gamma)\\
&=\sum_{|\gamma|=d-1}\left(\sum_{e_i\nmid\gamma}(-1)^{(e_i\gamma)}f_i\left(\sum_{\alpha\wedge\beta=\pm[e_i\gamma]}(-1)^{|\beta|+(\alpha\beta)}t^\beta t^\alpha\right)\right)\otimes\gamma
\end{align*}
We want to show the above quantity is zero.  For this it suffices to show that for $\gamma\in\mathscr B$ with
$|\gamma|=d-1$,
\[
\sum_{e_i\nmid\gamma}(-1)^{(e_i\gamma)}f_i\left(\sum_{\alpha\wedge\beta=\pm[e_i\gamma]}(-1)^{|\beta|+(\alpha\beta)}t^\beta t^\alpha\right)=0
\]
We have
\begin{equation}\begin{split}\label{q1}
\sum_{e_i\nmid\gamma}(-1)^{(e_i\gamma)}f_i&\left(\sum_{\alpha\wedge\beta=\pm[e_i\gamma]}(-1)^{|\beta|+(\alpha\beta)}t^\beta t^\alpha\right)=\\
&\hskip .5in\sum_{\alpha\wedge\beta=\pm\gamma}\left[t^\alpha\left(\sum_{e_i\nmid\gamma}(-1)^{|\alpha|+([e_i\beta]\alpha)+(e_i\gamma)}f_it^{[e_i\beta]}\right)\right.\\
&\hskip 1in +\left.\left(\sum_{e_i\nmid\gamma}(-1)^{|\beta|+1+(\alpha[e_i\beta])+(e_i\gamma)}f_it^{[e_i\beta]}\right)t^\alpha\right]
\end{split}\end{equation}
From the equalities 
\begin{align*}
[e_i\gamma]&=(-1)^{(e_i\gamma)}e_i\wedge\gamma\\
&=(-1)^{(e_i\gamma)+(\beta\alpha)}e_i\wedge\beta\wedge\alpha\\
&=(-1)^{(e_i\gamma)+(\beta\alpha)+(e_i\beta)}[e_i\beta]\wedge\alpha\\
&=(-1)^{(e_i\gamma)+(\beta\alpha)+(e_i\beta)+([e_i\beta]\alpha)}[e_i\gamma]
\end{align*}
we see that $(-1)^{([e_i\beta]\alpha)+(e_i\gamma)}=(-1)^{(\beta\alpha)+(e_i\beta)}$.
Quantity \eqref{q1} above thus becomes
\begin{equation}\begin{split}\label{q2}
\sum_{\alpha\wedge\beta=\pm\gamma}&\left[t^\alpha\left(\sum_{e_i\nmid\gamma}(-1)^{|\alpha|+(\beta\alpha)+(e_i\beta)}f_it^{[e_i\beta]}\right)\right.\\
&\hskip 1in +\left.\left(\sum_{e_i\nmid\gamma}(-1)^{(|\alpha|+1)(|\beta|+1)+(\beta\alpha)+(e_i\beta)}f_it^{[e_i\beta]}\right) t^\alpha\right]
\end{split}\end{equation}
Since $|[e_i\beta]|\le d$, we have by induction the relations
\[
\sum_{\delta\wedge\epsilon=\pm\beta}(-1)^{|\epsilon|+(\delta\epsilon)}t^\epsilon t^\delta
+\sum_{e_i\nmid\beta}(-1)^{|\beta|+(e_i\beta)}f_it^{[e_i\beta]}=0
\]
which we may rearrange to
\[
\sum_{e_i\nmid\gamma}(-1)^{(e_i\beta)}f_it^{[e_i\beta]}=\sum_{e_i|\gamma\atop e_i\nmid\beta}(-1)^{(e_i\beta)+1}f_it^{[e_i\beta]}+\sum_{\delta\wedge\epsilon=\pm\beta}(-1)^{|\delta|+(\delta\epsilon)+1}t^\epsilon t^\delta
\]
Substituting this expression into \eqref{q2} gives
\begin{equation}\begin{split}\label{q3}
\sum_{\alpha\wedge\beta=\pm\gamma}&\left[t^\alpha\left(\sum_{e_i|\gamma\atop e_i\nmid\beta}(-1)^{|\alpha|+(\beta\alpha)+(e_i\beta)+1}f_it^{[e_i\beta]}\right)\right.\\
&\hskip 1in +\left.\left(\sum_{e_i|\gamma\atop e_i\nmid\beta}(-1)^{(|\alpha|+1)(|\beta|+1)+(\beta\alpha)+(e_i\beta)+1}f_it^{[e_i\beta]}\right) t^\alpha\right]\\
+\sum_{\alpha\wedge\beta=\pm\gamma}&\left[t^\alpha\left(\sum_{\delta\wedge\epsilon=\pm\beta}(-1)^{|\alpha|+(\beta\alpha)+|\delta|+(\delta\epsilon)+1}t^\epsilon t^\delta\right)\right.\\
&\hskip 1in +\left.\left(\sum_{\delta\wedge\epsilon=\pm\beta}(-1)^{(|\alpha|+1)(|\beta|+1)+(\beta\alpha)+|\delta|+(\delta\epsilon)+1}t^\epsilon t^\delta\right) t^\alpha\right]
\end{split}\end{equation}
We will show that the two outer sums in \eqref{q3} are zero. The first outer sum may be broken up as
\begin{equation}\begin{split}\label{q4}
\sum_{\alpha\wedge\beta=\pm\gamma}&t^\alpha\left(\sum_{e_i|\gamma\atop e_i\nmid\beta}(-1)^{|\alpha|+(\beta\alpha)+(e_i\beta)+1}f_it^{[e_i\beta]}\right)\\
&\hskip 1in +\sum_{\alpha\wedge\beta=\pm\gamma}\left(\sum_{e_i|\gamma\atop e_i\nmid\beta}(-1)^{(|\alpha|+1)(|\beta|+1)+(\beta\alpha)+(e_i\beta)+1}f_it^{[e_i\beta]}\right) t^\alpha\\
\end{split}\end{equation}
Consider a term from the first sum:
\begin{equation}\label{q5}
(-1)^{|\alpha|+(\beta\alpha)+(e_i\beta)+1}f_it^\alpha t^{[e_i\beta]}
\end{equation} 
Since $e_i\nmid \beta$ and $e_i |\gamma$, we must have $e_i |\alpha$.  Therefore there is a corresponding term in the second sum:
\begin{equation}\label{q6}
(-1)^{(|\alpha'|+1)(|\beta'|+1)+(\beta'\alpha')+(e_i\beta')+1}f_it^{[e_i\beta']}t^{\alpha'}
\end{equation}
where $\alpha=[e_i\beta']$ and $\alpha'=[e_i\beta]$.  An easy calculation shows that 
$(-1)^{(\beta\alpha)+(e_i\beta)}=(-1)^{(\beta'\alpha')+(e_i\beta')+|\beta'|+|\alpha||\beta|}$.
Substituting this into \eqref{q5}, we see that the two terms in \eqref{q5} and \eqref{q6} are of opposite sign, and therefore cancel.  It follows that the entire quantity \eqref{q4} is zero.

The second outer sum in \eqref{q4} may be broken up as
\begin{equation}\begin{split}\label{q7}
\sum_{\alpha\wedge\beta=\pm\gamma}&t^\alpha\left(\sum_{\delta\wedge\epsilon=\pm\beta}(-1)^{|\alpha|+(\beta\alpha)+|\delta|+(\delta\epsilon)+1}t^\epsilon t^\delta\right)\\
&\hskip 1in +\sum_{\alpha\wedge\beta=\pm\gamma}\left(\sum_{\delta\wedge\epsilon=\pm\beta}(-1)^{(|\alpha|+1)(|\beta|+1)+(\beta\alpha)+|\delta|+(\delta\epsilon)+1}t^\epsilon t^\delta\right) t^\alpha
\end{split}\end{equation}
Consider a term of the first sum:
\begin{equation}\label{q8}
(-1)^{|\alpha|+(\beta\alpha)+|\delta|+(\delta\epsilon)+1}t^\alpha t^\epsilon t^\delta
\end{equation}
There is a corresponding term in the second sum:
\begin{equation}\label{q9}
(-1)^{(|\alpha'|+1)(|\beta'|+1)+(\beta'\alpha')+|\delta'|+(\delta'\epsilon')+1}t^{\epsilon'} t^{\delta'} t^{\alpha'}
\end{equation}
where $\alpha=\epsilon'$, $\epsilon=\delta'$ and $\delta=\alpha'$. An easy calculation shows that
$(-1)^{(\beta'\alpha')+(\delta'\epsilon')}=(-1)^{|\alpha||\delta|+|\delta||\epsilon|+(\beta\alpha)+(\delta\epsilon)}$.  Substituting this into \eqref{q9} shows that the two terms \eqref{q8} and \eqref{q9} are of opposite sign, and therefore cancel.  It follows that the entire quantity \eqref{q7} is zero.

We have now shown that 
\[
\partial^{\Hom_{\text gr}(F,F)\otimes K^\sharp}\left(\sum_{|\gamma|=d}\left(\sum_{\alpha\wedge\beta=\pm\gamma}(-1)^{|\beta|+(\alpha\beta)}t^\beta t^\alpha\right)\otimes\gamma\right)=0
\]
By exactness of $\Hom_{\text gr}(F,F)\otimes K^\sharp$, there exists for each $\mu\in\mathscr B$ with $|\mu|=d+1$,  endomorphisms $t^\mu\in\Hom_{\text gr}(F,F)$ of degree $-d-2$ such that
\[
\partial^{\Hom_{\text gr}(F,F)\otimes K^\sharp}\left((-1)^{d+1}\sum_{|\mu|=d+1}t^\mu\otimes\mu\right)=
\sum_{|\gamma|=d}\left(\sum_{\alpha\wedge\beta=\pm\gamma}(-1)^{|\beta|+(\alpha\beta)}t^\beta t^\alpha\right)\otimes\gamma
\]
It is easy to see that \eqref{relation} holds for each $\gamma\in\mathscr B$, $|\gamma|=d$, and thus induction is complete.
\end{proof}

Consider the complex $F\otimes_Q K$. We perturb its differential 
$\partial^F\otimes K + F\otimes\partial^K$ to 
\[
\partial=\sum_{\alpha\in\mathscr{B}}t^\alpha\otimes s_\alpha
\]
where $t^0=\Id_F$, $t^1=\partial^F$, $s_0=\partial^K$,
$t^\alpha$ are defined by Lemma \ref{lemma} for $\alpha\ne 0$, and $s_\alpha$ is multiplication by $\alpha$
for $\alpha\ne 0$.  

\begin{lemma}
\[
\partial^2=0
\]
\end{lemma}

\begin{proof} Let $x\otimes y\in F\otimes_QK$.  Then
\begin{align*} 
\partial^2(x\otimes y)=&\left(\sum_{\beta\in\mathscr{B}}t^\beta\otimes s_\beta\right)\left(\sum_{\alpha\in\mathscr{B}}t^\alpha\otimes s_\alpha\right)(x\otimes y)\\
=&\left(\sum_{\beta\in\mathscr{B}}t^\beta\otimes s_\beta\right)\left(\sum_{\alpha\ne 0}
(-1)^{|x||\alpha|}t^\alpha(x)\otimes \alpha\wedge y+(-1)^{|x|}x\otimes\partial^K(y)\right)\\
=&\left(\sum_{\beta\ne 0}t^\beta\otimes s_\beta+t^0\otimes s_0\right)\left(\sum_{\alpha\ne 0}
(-1)^{|x||\alpha|}t^\alpha(x)\otimes \alpha\wedge y+(-1)^{|x|}x\otimes\partial^K(y)\right)\\
=&\sum_{\alpha,\beta\ne 0}(-1)^{|x|(|\alpha|+|\beta|)+(|\alpha|+1)|\beta|}t^\beta 
t^\alpha(x)\otimes\beta\wedge\alpha\wedge y\\
&+\sum_{\beta\ne 0}(-1)^{|x||\beta|+|x|}t^\beta(x)\otimes \beta\wedge\partial^K(y)\\
&+\sum_{\alpha\ne 0}(-1)^{|x||\alpha|+|x|-|\alpha|-1}
t^\alpha(x)\otimes \partial^K(\alpha\wedge y)\\
\end{align*}
Now using the Leibniz rule $\partial^K(\alpha\wedge y)=
\partial^K(\alpha)\wedge y+(-1)^{|\alpha|}\alpha\wedge\partial^K(y)$ in the third sum, we achieve
cancellation in the second and third sums to obtain
\begin{align*} 
\partial^2(x\otimes y)=&\sum_{\alpha,\beta\ne 0}(-1)^{|x|(|\alpha|+|\beta|)+(|\alpha|+1)|\beta|}t^\beta 
t^\alpha(x)\otimes\beta\wedge\alpha\wedge y\\
&+\sum_{|\alpha|>0}(-1)^{|x||\alpha|+|x|-|\alpha|-1}
t^\alpha(x)\otimes \partial^K(\alpha)\wedge y\\
=&\left(\sum_{\alpha,\beta\ne 0}(-1)^{(|\alpha|+1)|\beta|}t^\beta 
t^\alpha\otimes\beta\wedge\alpha\right.\\
&+\left.\sum_{|\alpha|>0}(-1)^{-|\alpha|-1}
t^\alpha\otimes \partial^K(\alpha)\right)(x\otimes y)\\
\end{align*}
We see then that for $\partial^2(x\otimes y)=0$ for arbitrary $x\otimes y$ we must have 
for all $\gamma\in\mathscr{B}$ 
\begin{align*}
\sum_{\alpha\wedge\beta=\pm\gamma}(-1)^{|\beta|+(\alpha\beta)}t^\beta t^\alpha
+\sum_{e_i\nmid\gamma}(-1)^{|\gamma|+(e_i\gamma)}f_it^{[e_i\gamma]}=0
\end{align*}
where $(-1)^{(\alpha\beta)}\alpha\wedge\beta=\gamma$ and $[e_i\gamma]$ is the element in $\mathscr{B}$
equal to $e_i\gamma$ up to sign.  But this is exactly what we get from Lemma \ref{lemma}.
\end{proof}

\begin{theorem}\label{main} 
The complex $(F\otimes_QK,\partial)$ of finitely generated free $Q$-modules has the same homology as 
$(\overline F,\partial^{\overline F})$.
\end{theorem}

\begin{proof}
We filter the complex $F\otimes_QK$ as
\[
\cdots\subset \mathscr F^{p-1} \subset \mathscr F^p \subset \cdots
\]
for $p\in\mathbb Z$, where $\mathscr F^p$ is the subcomplex of $F\otimes_QK$ given by $\mathscr F^p=\coprod_{i\le p}F_i\otimes_Q K$.
We note that since $F\otimes_QK$ is a horizontal band in the upper half plane, this filtration is bounded.  
Therefore we have a convergent spectral sequence $E^2_{p,q} \underset{p}{\Rightarrow} \H_n(F\otimes_QK)$. In
fact, for each $n\in\mathbb Z$, $\mathscr F^{n-c-1}_n=0$ and $\mathscr F^n_n=(F\otimes_QK)_n$.  Thus for the induced 
filtration $(\Phi^p\H_n(F\otimes_QK))$ of $\H_n(F\otimes_QK)$ we have
\[
0=\Phi^{n-c-1}H_n\subseteq \Phi^{n-c}H_n \subseteq \cdots \subseteq \Phi^{n}H_n=\H_n
\]

Since the quotients $\mathscr F^p/\mathscr F^{p-1}$ are isomorphic to $F_p\otimes K$, we see that the $E^1_{p,q}$
terms of the spectral sequence are
\[
E^1_{p,q}\cong\begin{cases} 0 & p>0\\
                        F_q\otimes_QR & p=0
												\end{cases}
\]
subsequently, the spectral sequence collapses on the $p$-axis, and the nonzero $E^2$ terms are given by the homology of $\overline F=F\otimes_QR$.  Specifically,
\[
E^2_{p,q}\cong\begin{cases} 0 & q\ne0\\
                        \H_p(\overline F) & q=0
												\end{cases}
\]	
It follows that 
\[
\H_n(F\otimes_QK)\cong\H_n(\overline F)
\]
and so $F\otimes_QK$ and $\overline F$ have the same homology.						
\end{proof}

\begin{corollary}
Suppose that $(\overline F,\partial^{\overline F})$ is an $R$-free resolution of a finitely generated 
$R$-module $M$.  Then $(F\otimes_QK,\partial)$ is a $Q$-free resolution of $M$.
\end{corollary}

\begin{remark}\label{Eisops}
We have from Lemma \ref{lemma} that the endomorphisms $t^{e_1},\dots,t^{e_c}$ satisfy
$t^1t^1+\sum f_it^{e_i}=0$, that is, 
\[
(\partial^F)^2=-\sum_{i=1}^c f_it^{e_i}
\]
Thus $t^{e_i}\otimes R$ is nothing more than the negative of the \emph{Eisenbud operator} 
$t_i(Q,\{f_i\},\overline F)$ defined in Section 1 of \cite{Eisenbud}.  One therefore may think of the
$t^\alpha\otimes R$ for $\alpha\in\mathscr B$ with $|\alpha|>1$ as \emph{higher-order} Eisenbud operators.

As is shown in {\it loc. cit.}, each $t^{e_i}\otimes R:\overline F\to\overline F$ is a chain map of degree $-2$. 
We remark that this fact follows directly from (\ref{relation}) when $\gamma=e_i$. Also from (\ref{relation}),
for $\gamma=e_1\wedge e_2$ for example, we have
\[
t^{e_1\wedge e_2}t^1 + t^1t^{e_1\wedge e_2}+t^{e_1}t^{e_2}-t^{e_2}t^{e_1}+
\sum_{i\ne 1,2}(-1)^{[e_i(e_1\wedge e_2)]}f_it^{e_1\wedge e_2\wedge e_i}=0
\]
From which it follows that $t^{e_1}\otimes R$ and $t^{e_2}\otimes R$ commute with each other up to homotopy,
with $t^{e_1\wedge e_2}$ being the homotopy.  the relations (\ref{relation}) show the same holds for all $t^{e_i}$ and $t^{e_j}$; this is \cite[Corollary 1.5]{Eisenbud}. 
\end{remark}

\section{Codimension one}

When $c=1$, in other words, when $R=Q/(f)$ for a single nonzerodivisor $f$, the complex over $Q$ we construct has a particularly simple form, and isomorphic incarnations of it have appeared in recent articles \cite{Steele}, 
\cite{BerghJorgensen}, \cite{AtkinsVraciu}. It is described as follows.

In each degree, say $n$, the complex $F\otimes_QK$ has just two summands, $F_{n-1}\otimes_QK_1$ and 
$F_n\otimes_QK_0$.  Since $K_1=Re$ and $K_0=R$ are rank one free modules, we may identify these two summands with $F_{n-1}$ and $F_n$, respectively.  The differential of $F\otimes_QK$ is just $t^0\otimes s_0+t^1\otimes s_1+t^e\otimes s_e$,
and tracking through the identifications above we see that $F\otimes_QK$ is isomorphic to 
\begin{align}
\cdots\to F_n \oplus F_{n+1} & \xrightarrow{\left[\begin{smallmatrix} \partial^F_n & (-1)^{n+1}t^e \\ 
(-1)^nf & \partial^F_{n+1} \end{smallmatrix}\right]} F_{n-1} \oplus F_n \xrightarrow{\left[\begin{smallmatrix} \partial^F_{n-1} & (-1)^nt^e \\ (-1)^{n-1}f & \partial^F_n \end{smallmatrix}\right]}  F_{n-2} \oplus F_{n-1}\to
\\
&\cdots\to F_1 \oplus F_2 \xrightarrow{\left[\begin{smallmatrix} \partial^F_1 & t^e \\ -f & \partial^F_2 
\end{smallmatrix}\right]}  F_0 \oplus F_1 \xrightarrow{\left[\begin{smallmatrix} f & \partial^F_1 
\end{smallmatrix}\right]}  F_0 \to 0  \nonumber
\end{align} 

\subsection*{Minimality} We say that a complex of finitely generated free $Q$-modules $(C,\partial^C)$ is
\emph{minimal} if $\Im\partial_n^C\subseteq\n C_{n-1}$ for all $n$.  Assuming that $\overline F$ is taken 
minimal, the question of whether $F\otimes_QK$ is minimal then boils down to whether $\Im t^e\subseteq\n F$.
One simple way this can happen is if $t^e_n$ is actually zero for all $n$:

\begin{discussion} Let $R=Q/(f)$ where $f$ is a nonzerodivisor of $Q$.  Let $\widetilde M$ be a maximal 
Cohen-Macaulay $Q$-module, and $F$ a minimal free resolution of $\widetilde M$ over $Q$.  Then 
$\overline F=F\otimes_QR$ is a minimal free resolution of $M=\widetilde M\otimes_QR$ over $R$. Choosing 
the lifting $F$ of $\overline F$ to $Q$, we see that $(\partial ^F)^2=0$, so that $t^e_n=0$ for all $n$.
It follows that $F\otimes_QK$ is minimal, in fact, it has the form
\[
\cdots\to F_2\oplus F_3 \xrightarrow{\left[\begin{smallmatrix} \partial^F_2 & 0 \\ f & \partial^F_3 \end{smallmatrix}\right]} F_1\oplus F_2 \xrightarrow{\left[\begin{smallmatrix} \partial^F_1 & 0 \\ -f & \partial^F_2 \end{smallmatrix}\right]} F_0\oplus F_1 
\xrightarrow{\left[\begin{smallmatrix}  f & \partial^F_1 \end{smallmatrix}\right]} F_0 \to 0
\]
(In this case we say that $M$ \emph{lifts} to $Q$.)

On the other hand, quite often $\Im t^e\nsubseteq\n F$. for example, assume that $Q$ is a regular local ring, 
and $f\in\n^2$, so that $R$ is a singular hypersurface ring.  Let $M$ be a maximal Cohen-Macaulay $R$-module.  
Then as described in \cite{Eisenbud}, a free resolution $\overline F$ of $M$ over $R$ can be chosen periodic of period $\le 2$, with lifting $F$ to $Q$ such that the maps $\partial^F$ comprise a matrix factorization of $f$.
It follows that the $t^e$ are identity maps, and so $F\otimes_QK$ is not minimal.  In fact, it is an infinite resolution over a regular local ring, and so cannot be minimal by Hilbert's Syzygy Theorem.

\end{discussion}

\section{Rank} In this section we assume that $R=Q/(f_1,\dots,f_c)$ for a regular sequence
$f_1,\dots,f_c$ of length $c$. Let $d$ denote the dimension of $Q$. 

\begin{discussion}
We see from the construction of $F\otimes_QK$ the following comparison between the 
ranks of the free modules of $\overline F$ versus those of $F\otimes_QK$:
\begin{equation}\label{binform}
\rank_Q(F\otimes_QK)_n=\sum_{i=0}^c {c \choose i} \rank_R\overline F_{n-i}
\end{equation}
Summing up we get
\begin{equation}\label{rank}
\sum_{n\ge 0}\rank_Q(F\otimes_QK)_n=\sum_{n\ge 0}\left(\sum_{i=0}^c{c\choose i}\rank_R\overline F_{n-i}\right)=2^c\sum_{n\ge 0}\rank_R\overline F_n
\end{equation}
Of course this equality is only interesting when $\overline F$ is a finite complex.  In this case we have the following application.  First we recall a well-known conjecture.
\end{discussion}

Buchsbaum and Eisenbud \cite[Proposition 1.4]{BuchsbaumEisenbud}, and Horrocks
\cite[Problem 24]{Hartshorne} conjectured that a free resolution $F$ of a nonzero module of finite length over a local ring $R$ of dimension $d$ satisfies $\rank_R F_i \ge {d \choose i}$ for all $i$. The weaker statement
that $\sum_{i\ge 0}\rank_R F_i \ge 2^d$ was later conjectured by Avramov \cite[pp. 63]{EvansGriffith}. We will refer to this latter conjecture as the \emph{total rank conjecture}; it was recently proved by Walker 
\cite[Theorem 1]{Walker} when R is a complete intersection whose residual characteristic is not two, and also when R is any local ring containing a field of positive characteristic not equal to two.  In our application below, we show that the conjecture holds modulo nonzerodivisors.

\subsection*{Application to rank conjectures}

\begin{theorem} Let $R=Q/(f)$, where $Q$ be a commutative local ring of dimension $d$, and $f$ is a nonzerodivisor contained in the maximal ideal of $Q$.  The total rank conjecture holds for $R$ if it does so for $Q$.
\end{theorem}

\begin{proof} Suppose that $M$ is an $R$-module with finite length having a finite free resolution $\overline F$ over $R$ with $\sum_{n\ge 0}\rank\overline F_n<2^{d-1}$.  The from (\ref{rank}) we see that $M$ has the finite
free resolution $F\otimes_QK$ over $Q$ with $\sum_{n\ge 0}\rank(F\otimes_QK)_n<2^d$.
\end{proof}

Looking at (\ref{binform}) we have the following refinement.

\begin{proposition}
Assume that $Q$ is a local ring of dimension $d$, and $f_1,\dots,f_c$ a $Q$-regular sequence contained in the maximal ideal of $Q$. Let $\overline F$ be a minimal $R$-free resolution of a finite length $R$-module $M$, and
$F\otimes_QK$ the resolution from Theorem \ref{main}.  If 
$\rank_R\overline F_{n-i}\ge {d-c\choose n-i}$ for $i=0,\dots,c$, then $\rank_Q(F\otimes_QK)_n\ge 
{d\choose n}$.
\end{proposition}

\begin{proof}
From (\ref{binform}) we have  
\[
\rank_Q(F\otimes_QK)_n=\sum_{i=0}^c {c \choose i} \rank_R\overline F_{n-i}\ge \sum_{i=0}^c {c \choose i}{d-c\choose n-i}={d\choose n}
\]
where the last equality is a standard identity.
\end{proof}

\section{Totally acyclic complexes}

Our interest in this project is due to the adjoint pair of functors 
\[
\xymatrixrowsep{2pc}
\xymatrixcolsep{3pc}
\xymatrix{
\Ktac(Q) \ar@{->}[r]<0.5 ex>^{S} & \Ktac(R) \ar@{->}[l]<0.5 ex>^{T}
}
\]
defined in \cite{BerghJorgensenMoore}, where $\Ktac(Q)$ is the homotopy category of totally acyclic
complexes over $Q$, and $\Ktac(R)$ is the homotopy category of those over $R$. The main point of 
\emph{loc. cit.} is that the adjoint pair of functors provides approximations of elements in 
$\Ktac(R)$ by those in the image of the functor $S$.  Specifically, for $C\in\Ktac(R)$ we have 
a morphism $\epsilon_C: STC \to C$, and this morphism is the approximation.  The functor $T$ is defined
on objects as follows: for $C\in\Ktac(R)$, $TC$ is a complete resolution of $\Im\partial_0^C$. The functor
$S$ is simply the base change functor $\underline{\quad}\otimes_QR$. The connection
between \cite{BerghJorgensenMoore} and the current project is expressed in the following theorem. 

\begin{theorem}
Assume that $R=Q/(f_1,\dots,f_c)$, where $f_1,\dots,f_c$ is a $Q$-regular sequence, and let $C\in\Ktac(R)$.  Letting $\overline F=C$, the complex $F\otimes_QK$ defined above is $TC$.  That is, $F\otimes_QK$ is a complete resolution of $\Im\partial_0^C$. 
\end{theorem}
It follows that $STC$ is $(F\otimes_QK)\otimes_QR$, and the morphism
\[
\epsilon_C:STC\to C
\]
is the map that projects the copy $\overline F\otimes_Q K_0$ of $\overline F$ in $(F\otimes_QK)\otimes_QR$
onto $\overline F=C$.

We illustrate the whole affair with an example.

\begin{example} Let $R=k[x,y]/(x^2,y^2)$, and $C$ be the totally acyclic $R$-complex with 
$\Im\partial^C_0=Rxy\cong k$:
\[
C=\overline F: \cdots \to R^3 \xrightarrow{\left[\begin{smallmatrix} x & 0 & -y \\ 0 & y & x \end{smallmatrix}\right]} R^2 \xrightarrow{\left[\begin{smallmatrix} x & y \end{smallmatrix}\right]} R \xrightarrow{\left[\begin{smallmatrix} xy \end{smallmatrix}\right]} R \xrightarrow{\left[\begin{smallmatrix} x \\ y \end{smallmatrix}\right]}R^2 \to\cdots
\]
A lifting of $C$ to a sequence of homomorphisms over $Q=k[x,y]/(x^2)$ is given by
\[
\widetilde C=F: \cdots \to Q^3 \xrightarrow{\left[\begin{smallmatrix} x & 0 & -y \\ 0 & y & x \end{smallmatrix}\right]} Q^2 \xrightarrow{\left[\begin{smallmatrix} x & y \end{smallmatrix}\right]} Q \xrightarrow{\left[\begin{smallmatrix} xy \end{smallmatrix}\right]} Q \xrightarrow{\left[\begin{smallmatrix} x \\ y \end{smallmatrix}\right]}Q^2 \to\cdots
\]
We have 
\begin{align*}
\partial_1^F\partial_2^F=&\left[\begin{matrix} x & y \end{matrix}\right]\left[\begin{matrix} x & 0 & -y \\ 0 & y & x \end{matrix}\right]=\left[\begin{matrix} 0 & y^2 & 0 \end{matrix}\right]=
y^2\left[\begin{matrix} 0 & 1 & 0 \end{matrix}\right]
\\
\partial_0^F\partial_1^F=&\left[\begin{matrix} xy \end{matrix}\right]\left[\begin{matrix} x & y \end{matrix}\right]=\left[\begin{matrix} 0 & xy^2 \end{matrix}\right]=
y^2\left[\begin{matrix} 0 & x \end{matrix}\right]
\\
\partial_{-1}^F\partial_0^F=&\left[\begin{matrix} x \\ y \end{matrix}\right]\left[\begin{matrix} xy \end{matrix}\right]=\left[\begin{matrix} 0 \\ xy^2 \end{matrix}\right]=
y^2\left[\begin{matrix} 0 \\ x \end{matrix}\right]
\end{align*}
We see that $t_2=\left[\begin{matrix} 0 & -1 & 0 \end{matrix}\right]$, $t_1=\left[\begin{matrix} 0 & -x \end{matrix}\right]$, and $t_0=\left[\begin{matrix} 0 \\ -x \end{matrix}\right]$, where we have written $t$ to represent $t^e$.  Therefore $TC$ is
\[
\cdots\to Q^2\oplus Q^3 \xrightarrow{\left[\begin{smallarray}{cc|ccc}  x & y & 0 & -1 & 0 \\ \hline -y^2 & 0 & x & 0 & -y \\ 0 & -y^2  &  0 & y & x \end{smallarray}\right]} Q^1\oplus Q^2 \xrightarrow{\left[\begin{smallarray}{c|cc} xy  & 0 & x  \\ \hline y^2 &  x & y \end{smallarray}\right]} Q^1\oplus Q^1
\xrightarrow{\left[\begin{smallarray}{c|c} x & 0 \\ y & -x \\ \hline -y^2 & xy \end{smallarray}\right]} Q^2 \oplus Q^1 \to \cdots
\]
This section of the right approximation $\epsilon_C:STC \to C$ thus takes the form
\[
\xymatrixrowsep{3pc}
\xymatrixcolsep{6pc}
\xymatrix{
R^2\oplus R^3\ar@{->}[r]^{\left[\begin{smallarray}{cc|ccc}  x & y & 0 & -1 & 0 \\ \hline -y^2 & 0 & x & 0 & -y \\ 0 & -y^2  &  0 & y & x \end{smallarray}\right]}\ar@{->}[d]^{\left(\begin{smallmatrix}  0 & 1 \end{smallmatrix}\right)} & R\oplus R^2 \ar@{->}[d]^{\left(\begin{smallmatrix}  0 & 1 \end{smallmatrix}\right)} \ar@{->}[r]^{\left[\begin{smallarray}{c|cc} xy  & 0 & x  \\ \hline y^2 &  x & y \end{smallarray}\right]} & R\oplus R
\ar@{->}[d]^{\left(\begin{smallmatrix}  0 & 1 \end{smallmatrix}\right)} \ar@{->}[r]^{\left[\begin{smallarray}{c|c} x & 0 \\ y & -x \\ \hline -y^2 & xy \end{smallarray}\right]} & R^2\oplus R \ar@{->}[d]^{\left(\begin{smallmatrix}  0 & 1 \end{smallmatrix}\right)} 
\\
R^3\ar@{->}[r]_{\left[\begin{smallmatrix} x & 0 & -y \\ 0 & y & x \end{smallmatrix}\right]} & R^2 \ar@{->}[r]_{\left[\begin{smallmatrix} x & y \end{smallmatrix}\right]} & R \ar@{->}[r]_{\left[\begin{smallmatrix} xy \end{smallmatrix}\right]} & R 
}
\]
\end{example}


\begin{thebibliography}{BeJoMo}
\bibitem[AtVr]{AtkinsVraciu} C.\ Atkins and A. \ Vraciu, \emph{On the existence of non-free totally reflexive modules}, math arXiv:1602.08385
\bibitem[Av]{Avramov}L.L.\ Avramov, \emph{Modules of finite virtual projective dimension}, Invent.\ Math.\ {\bf 96} (1989), no.\ 1, 71--101.
\bibitem[BeJo]{BerghJorgensen} P.\ A. \ Bergh and D.\ A. \ Jorgensen, \emph{A generalized Dade's Lemma for local rings}, Algebr. Represent. Theory (to appear)
\bibitem[BeJoMo]{BerghJorgensenMoore} P.\ A. \ Bergh, D.\ A. \ Jorgensen and W.\ F.\ Moore, \emph{Totally acyclic
approximations}, in preparation.
\bibitem[Bu]{Burke} J.\ Burke, \emph{Homological Koszul duality for complete intersection rings}, in preparation.
\bibitem[BucEi]{BuchsbaumEisenbud} D.\ Buchsbaum and D. \ Eisenbud, \emph{Algebra structures for finite free resolutions, and some structure theorems for ideals of codimension 3}, Amer. J. Math., {\bf 99} (1977), 447--485.\bibitem[Ch]{Chase} U.\ Chase, \emph{Direct products of modules}, Trans.\ Amer. \ Math. \ Soc.\ {\bf 97} (1960), 457--473.
\bibitem[Ei]{Eisenbud} D.\ Eisenbud, \emph{Homological algebra on a complete intersection, with an application to group representations}, Trans.\ Amer.\ Math.\ Soc.\ {\bf 260} (1980), no.\ 1, 35--64.
\bibitem[EiPeSc]{EisenbudPeevaSchreyer} D.\ Eisenbud, I.\ Peeva and F.-O.\ Schreyer \emph{Tor as a module over an exterior algebra}, arXiv:1810.04999
\bibitem[EvGr]{EvansGriffith}   E.\ G. Evans and P.\ Griffith, \emph{Syzygies}, London Math. Soc. Lecture Note Ser.,
{\bf 106}, Cambridge University Press, Cambridge, 1985.
\bibitem[Ha]{Hartshorne} R.\ Hartshorne, \emph{Algebraic vector bundles on projective space:  A problem list}, Topology, {\bf 18} (1979), 117--128.
\bibitem[Na]{Nagata} M.\ Nagata, \emph{Local rings}, Wiley, New York, 1962.
\bibitem[St]{Steele} N.T.\ Steele, \emph{Support and rank varieties of totally acyclic complexes}, Ph.D. thesis, University of Texas Arlington, 2016. 
\bibitem[Wa]{Walker} M.\ E.\ Walker, \emph{Total Betti numbers of modules of finite projective dimension}, Ann. of Math., {\bf 186} (2017), 641--646.
\end{thebibliography}
\end{document}